 \documentclass[a4paper,10pt]{amsart}
\usepackage[matrix,arrow,tips,curve]{xy}
\usepackage[english]{babel}
\usepackage[leqno]{amsmath}
\usepackage{amssymb,amsthm}
\usepackage{amscd}
\usepackage{enumerate}
\usepackage{color}
\input{xy}
\xyoption{all}

\newtheorem{thm}{Theorem}[section]
\newtheorem{lemma}[thm]{Lemma}

\newtheorem{prop}[thm]{Proposition}

\newtheorem{claim}[thm]{Claim}

\theoremstyle{remark}
\newtheorem{remark}[thm]{Remark}

\theoremstyle{definition}
\newtheorem{defi}[thm]{Definition}

\newtheorem{example}[thm]{Example}

\newcommand{\la}{\longrightarrow}
\newcommand{\ha}{\hookrightarrow}

\newcommand{\ov}{\overline}

\newcommand{\codim}{\operatorname{codim}}
\newcommand{\Div}{\operatorname{Div}}

\newcommand{\supp}{\operatorname{Supp}}

\newcommand{\Pic}{\operatorname{Pic}}
\newcommand{\Tw}{\operatorname{Tw}}
\newcommand{\Jac}{\operatorname{Jac}}

\newcommand{\Prin}{\operatorname{Prin}}
\newcommand{\ord}{\operatorname{ord}}

\newcommand{\Picphi}{\operatorname{Pic}_{\phi}}

\newcommand{\mdeg}{\underline{\operatorname{deg}}}
\newcommand{\val}{\operatorname{val}}

\def\e{\epsilon}

\def\L{\mathcal L}
\def\O{\mathcal O}

\def\X{\mathcal X}

\newcommand{\Z}{\mathbb{Z}}
\newcommand{\R}{\mathbb{R}}
\newcommand{\N}{\mathbb{N}}

\newcommand{\G}{G}

\def\mo{\underline{0}}

\def\mn{\underline{n}}

\newcommand{\dv}{{\rm div}}

\newcommand{\Mgt}{{M_g^{\rm trop}}}

\newcommand{\Mgb}{\ov{M_g}}

\newcommand{\lp}{\operatorname{loop}}
\newcommand{\hG}{\widehat{G}}
\newcommand{\hGw}{\widehat{G^{\omega}}}
\newcommand{\tG}{\widetilde{G}}
\newcommand{\hD}{\widehat{D}}

\newcommand{\hv}{\widehat{v}}
\newcommand{\tv}{\widetilde{v}}

\newcommand{\te}{\widetilde{e}}

\renewcommand{\div}{\mathrm{div}}

\newcommand{\Ks}{K^{\#}_{\G}}
\newcommand{\rs}{r^{\#}}
\begin{document}
\title{Riemann-Roch theory    for weighted graphs and tropical curves}
\author{Omid Amini and   Lucia Caporaso}
\address{D\'epartement de math\'ematiques et applications. \'Ecole Normale Sup\'erieure, 
  45 Rue d'Ulm, 75005 Paris (France).} \email{oamini@math.ens.fr}
 \address{Dipartimento di Matematica e Fisica,
 Universit\`a Roma Tre,
 Largo S. Leonardo Murialdo 1,
 00146 Roma (Italy).}
 \email{caporaso@mat.uniroma3.it}
%
\maketitle
\begin{abstract}
We define a divisor theory for graphs and tropical curves endowed with a weight function on the vertices;  we prove that the Riemann-Roch theorem holds in both cases.
We extend   Baker's Specialization Lemma to weighted graphs.
\end{abstract}

\tableofcontents
\section{Introduction}

The notion of vertex weighted graph,  i.e. a graph whose vertices are assigned a non negative integer (the weight), arises naturally in algebraic geometry, as  every
Deligne-Mumford stable curve has an associated weighted ``dual" graph,
and the moduli space of   stable curves, $\Mgb$,  has a stratification with nice properties given by  the loci of curves having a certain weighted graph as dual graph; see \cite{gac}.

On the other hand, and more recently, vertex weighted graphs have appeared in tropical geometry  in the study of  degenerations of tropical curves obtained by letting the lengths of some edges go to zero.
To describe the limits of such families, with the above algebro-geometric picture in mind, one is led to consider metric graphs with a weight function on the vertices keeping track of the cycles that have vanished in the degeneration. Such metric weighted graphs are called weighted tropical curves;
they admit a moduli space, $\Mgt$, whose topological properties have strong similarities with those of $\Mgb$; see \cite{BMV} and \cite{CHBK}.

The connections between the algebraic and the tropical theory of curves have been the subject of much attention in latest times, and the topic presents a variety of interesting open problems. 
Moreover, the combinatorial skeleton of the theory,   its graph-theoretic side, has been studied in the weightless case independently of the tropical structure;   also in this setting the analogies with the
classical theory of algebraic curves are quite compelling; see \cite{BN} and \cite{BN2}.

In this paper we are interested in divisor theory. 
For graphs and tropical curves with no weights   the theory has been founded so that there are good notions of linear equivalence, canonical divisor, and rank of a divisor. One of the most important facts, as in algebraic geometry, is the Riemann-Roch theorem for the rank, which has been proved in \cite{BN} for loopless, weightless graphs, and in \cite{GK} and \cite{MZ} for weightless tropical curves.

The combinatorial theory is linked to the algebro-geometric theory not only by the formal analogies. Indeed, 
a  remarkable fact that connects the two theories is Baker's Specialization Lemma, of \cite
{bakersp}.
This result has been    applied in \cite{CDPR} to obtain a new proof of the famous Brill-Noether theorem for algebraic curves, in \cite
{bakersp} to prove   the Existence theorem (i.e., the non-emptyness of the Brill-Noether loci
when the Brill-Noether number is non-negative)    for weightless tropical curves,
and in \cite{CBNgraph},   strengthened by generalizing to graphs admitting loops (corresponding to the 
situation where the irreducible components of the special fiber could have nodal 
singularities),  to prove   the Existence theorem   for weightless graphs.
A   Specialization Lemma valid also for weighted graphs could be applied to relate the Brill-Noether loci of $\Mgb$ with those of $\Mgt$, or to characterize singular stable curves  that lie in the  Brill-Noether loci
(a well known open problem).  

The main goal of this paper is to set up the divisor theory for weighted graphs and tropical curves, and to  extend     the above results.  We hope in this way to prompt future developments
in tropical Brill-Noether theory; see  \cite{len}, for example.

We begin by giving a geometric interpretation of 
the weight structure; namely, we associate to every weighted graph a certain weightless graph, and to every weighted  tropical curve what we call a ``pseudo-metric" graph. In both cases, the weight of a vertex is given  a geometric interpretation using  certain ``virtual" cycles attached to that vertex;
in the tropical case such cycles have length zero,
so that  weighted tropical curves  bijectively correspond to pseudo-metric graphs; see Proposition~\ref{pseudo}.
Intuitively, from the algebro-geometric point of view where a graph is viewed as the dual graph of an algebraic curve,  the operation of adding virtual loops at a
vertex  corresponds to 
degenerating  the
  irreducible component corresponding to that vertex to a rational curve with a certain
number (equal to the weight of the vertex) of nodes,
 while breaking a loop  by inserting a new vertex 
translates, as in the weightless case, into  ``blowing up"   the node corresponding to the loop.

With these definitions  we prove that the Riemann-Roch theorem holds;
see Theorem~\ref{wRR} for graphs, and Theorem~\ref{RRwc} for tropical curves.  
Furthermore, we prove, in Theorem~\ref{spe}, that   
the   Specialization Lemma  holds   in a more general form taking into account the weighted structure. 
We note that this is a stronger fact than the   specialization lemma 
for weightless graphs~\cite{BN, CBNgraph}.   
For example, in the simplest case of a weighted graph consisting of a 
unique vertex without any edge, the inequalities of~\cite{BN, CBNgraph} become trivial, 
while the weighted specialization theorem we prove in this paper is equivalent to Clifford's inequality 
for irreducible curves. Moreover, one   easily sees that 
the operation of adding loops can only result in 
decreasing the rank of a given divisor, so our weighted specialization lemma gives 
stronger inequalities and more information on degeneration of line bundles. 
In fact, the proof of our result is not a simple consequence of the weightless case, and the argument requires some non-trivial algebro-geometric steps.

\ 

We wish to express our gratitude to Matt Baker for many stimulating conversations about  the contents of this paper, and to Yoav Len for pointing out a gap in an earlier proof of Theorem~\ref{spe}.
We also thank the referee for a very  accurate report.

\section{Preliminaries}

 \subsection{Divisor theory on graphs}
 \label{graphprel}
Graphs are assumed to be connected, unless otherwise specified.
We here extend the set-up of  \cite{BN} and \cite{bakersp} 
to graphs with loops. Our notation is non-sensitive to the presence or non-presence of loops.
 
Let $G$ be a graph and
  $V(\G)$ the set of its vertices.
  The group of divisors of $\G$, denoted by $\Div (\G)$, is the free abelian group generated by $V(G)$:
  $$
\Div (\G):=\{\sum _{v\in V(\G)}n_vv,\  n_v\in  \Z \}. 
  $$
For $D\in \Div (\G)$ we write
 $D=\sum _{v\in V(\G)}D(v)v
 $ where $D(v)\in \Z$.
For example, if $D=v_0$ for some $v_0\in V(\G)$, we have
\begin{displaymath}
v_0(v)=\left\{ \begin{array}{ll}
1  \text{ if }   v= v_0\\
0 \text{ otherwise.}\\
\end{array}\right.
\end{displaymath}

The degree of a divisor $D$ is  $\deg D:=\sum_{v\in V(\G)} D(v)$.
We say that $D$ is {\it effective}, and write $D\geq 0$, if $D(v)\geq 0$ for all $v\in V(\G)$.
We denote by 
 $\Div_+(\G)$ the set of effective divisors, and by 
$\Div ^d(\G)$ (respectively $\Div_+ ^d(\G)$) the set of divisors (resp. effective divisors) of degree $d$.   
  
  Let $\G$ be a graph and $\iota:H\ha \G$ a subgraph, so that we have
$V(H)\subset V(\G)$.
 For    any $D\in \Div(\G)$  we denote by $D_{H}\in \Div(H)$ 
 the restriction of $D$ to $H$.
 We have a natural injective  homomorphism
\begin{equation}
\label{notres}
\iota_*: \Div(H)\la \Div (G); \  \  D\mapsto \iota_*D
  \end{equation}
  such that $\iota_*D(v)=D(v)$ for every $v\in V(H)$ and $\iota_*D(u)=0$ for every $v\in V(G)\smallsetminus V(H)$.
 
\

\noindent
{\it Principal divisors.}
We shall now define principal divisors and linear equivalence. We set
\begin{displaymath}
(v\cdot w )=\left\{ \begin{array}{ll}
\text{number of edges joining }v  \text{ and } w&\text{ if } v\neq w\\
- \val(v) +2 \lp(v)&\text{ if } v= w\\
\end{array}\right.
\end{displaymath}
where  $\val(v)$ is the valency of  $v$, and $\lp(v)$ is the number of loops based at $v$.
This extends linearly to a 
 symmetric, bilinear ``intersection" product  
$$
\Div(G) \times \Div(G) \la \Z.
$$
Clearly, this product does not change if some loops are removed from $G$.

\medskip

For a vertex $v$ of $G$ we denote by 
$T_v\in \Div(G)$
the following divisor
$$
T_v :=\sum_{w\in V(\G)}(v\cdot w )w.
$$ 
Observe that $\deg T_v=0$.

\medskip
The group  $\Prin (\G)$ of {\it principal} divisors of $G$ is  the subgroup  of
 $\Div(\G)$ generated by
all the $T_v$:
$$
\Prin(\G)=<T_v,\  \forall v\in V(\G)>.
$$
We refer to the divisors $T_v$ as the {\it generators} of $\Prin (\G)$.

\medskip
For any subset $Z\subset V(G)$ we denote by  $T_Z\in \Prin (\G)$
 the divisor
\begin{equation}
\label{defT}
T_Z:=\sum_{v\in Z}T_v.
\end{equation}

 \begin{remark}
\label{oneless}
For any subset $U\subset V(\G)$ such that $|U|=|V(\G)|-1$
 the set $\{T_v,\  v\in U\}$ freely generates $\Prin (\G)$.
\end{remark}
Let us show that the above definition of principal divisors  coincides with the one given in \cite{BN}.
Consider  the set 
$  
k(\G):=\{f:V(\G) \to \Z \} 
 $ of integer valued functions on $V(\G)$.
Then
the divisor associated to $f$ is defined in \cite{BN} as 
$$
\dv (f):=   \sum_{v\in V(G)}\sum_{e=vw\in E(G)}(f(v)-f(w))v,
$$
and these are defined as the  principal divisors  in \cite{BN}. Now,
we have
\begin{eqnarray*}
  \dv (f) \,\, =& \sum_{v\in V(G)}\bigr(\sum_{w\in V(G)\smallsetminus
  v}(f(v)-f(w))(v\cdot w )\bigr)v\\
  \\
  =&\sum_{v\in V(G)}\Bigr[\Bigr(\sum_{w\in V(G)\smallsetminus v}(-f(w)
 (v\cdot w ))\Bigl) -f(v)(v\cdot v)\Bigl]v\\
  \\ =&-\sum_{v\in V(\G)}\Bigr(\sum_{w\in V(G)}f(w) (v\cdot w )\Bigl)v.
 \end{eqnarray*}

\noindent Fix any $v\in V(G)$ and consider the function $f_{v}:V(\G)\to \Z$ 
such that $f_{v}(v)=1$ and $f_{v}(w)=0$ for all $w\in V(\G)\smallsetminus v$.
 Using the above expression for $\dv (f)$ one checks that  
 $
T_{v}=-\dv (f_{v}).
$ As the functions $f_v$ generate $k(G)$, and the divisors $T_v$ generate $\Prin (G)$,
the two definitions of principal divisors are equal.

 \medskip

We say that $D,D'\in \Div (\G)$ are {\it linearly  equivalent}, and write $D\sim D'$, if $D-D'\in \Prin (\G)$.
We denote by
$ 
\Jac ^d(\G)=\Div^d(\G)/\sim$ the set of linear equivalence classes of divisors of degree $d$;
 we set
$$ 
\Jac(\G)= \Div (\G)/\Prin (\G).
$$ 
\begin{remark}
\label{connect}
If $d=0$ 
then $ 
\Jac ^0(\G)$ is a finite group, usually called the {\it Jacobian group} of $\G$.
This group has several other incarnations, most notably 
in combinatorics and algebraic geometry.
We need to explain the conection with \cite{cner}.
If $X_0$ is a nodal curve with dual graph $\G$
(see section~\ref{secspe}), the   elements of $\Prin (\G)$ correspond to the
multidegrees of some distinguished divisors  on $X_0$, called {\it twisters}. This explains why we denote by a decorated ``$T$" the elements of $\Prin (\G)$.
See \ref{connect2} for more details.
The Jacobian group $ 
\Jac ^0(\G)$ is the
same as the {\it degree class group} $\Delta_X$ of \cite{cner}; similarly, we have 
$ 
\Jac ^d(\G)=\Delta^d_X.
$ 
\end{remark}

Let $D\in \Div(\G)$;
in analogy with algebraic geometry, one denotes by 
$$ 
 |D\ |:=\{E\in \Div _+(\G): E\sim D\}
$$ the set of effective divisors equivalent to $D$.
Next, 
the  {\it   rank}, $r_{G}(D)$, of $D\in \Div(\G)$ is defined as follows. 
If $ |D|=\emptyset $ we set $r_{G}(D)=-1$. Otherwise we define
\begin{equation} 
\label{rank}
r_{G}(D):=\max\{k\geq 0: \  \forall  E\in \Div^k_+(\G) \   \  |D-E| \neq \emptyset \}.
\end{equation}

\begin{remark}
 The following facts follow directly from the definition.

\noindent If $D\sim D'$, then $r_{G}(D)=r_{G}(D')$.

\noindent If $\deg D<0$, then $r_{G}(D)=-1$.
Let $\deg D=0$; then $r_{G}(D)\leq 0$ with equality if and only if $D\in \Prin (\G)$.
\end{remark}
 
\

\noindent {\it Refinements of graphs.}
Let $\tG$ be a graph obtained  by adding a finite set of vertices in the interior of some of the edges of $\G$. We say that $\tG$ is a {\it refinement} of $\G$.
We have a natural inclusion $V(\G)\subset V(\tG)$; denote by $U:=V(\tG)\smallsetminus V(\G)$
the {\it new} vertices of $\tG$.
We have a natural   map
\begin{equation}
\label{notref}
\sigma^*:\Div(G)\la \Div (\tG);\  \  D\mapsto \sigma^*D
\end{equation}
such that $\sigma^*D(v)=D(v)$ for every $v\in V(G)$ and $\sigma^*D(u)=0$ for every $u\in U$.
 It is clear that $\sigma^*$ induces an isomorphism of $ \Div(G)$ with the subgroup of divisors on $\tG$ that vanish on $U$.
 The notation $\sigma^*$ is motivated in remark~\ref{pushrk}.
 
A particular case that we shall use a few times is that of a refinement of $G$ obtained by adding the same number,
$n$,
of vertices in the interior of every edge;
we denote by $G^{(n)}$ this graph, and refer to it as the {\it $n$-subdivision} of $G$.

\begin{remark}
\label{pushrk} 
Let $G$ be a graph and $e\in E(G)$ a fixed edge.
Let   $\tG$ be the refinement    obtained   by inserting only one vertex,
$\tv$,
in the interior   $e$. Let $v_1,v_2\in V(G)$ be the end-points of $e$, so that they are  also vertices of $\tG$. Note that $\tG$ has a unique edge $\te_1$ joining $v_1$ to $\tv$, and a unique edge  $\te_2$ joining $v_2$ to $\tv$. Then the contraction
of, say,  $\te_1$   is a   morphism of graphs
  $$
 \sigma :\tG \la G.
 $$
There is a natural pull-back map $ \sigma^*:\Div(G)\to \Div(\tG)$ associated to $\sigma$,
which maps $D\in \Div(G)$ to $\sigma^*D\in \Div(\tG)$ such that
$\sigma^*D(\tv)=0$, and $\sigma^*D$ is equal to $D$ on the remaining vertices of $\tG$,
which are of course identifed with the vertices of $G$.

\noindent By iterating, this construction generalizes to any refinement of $G$.

From this description, we have that the map $\sigma^*$ coincides with the map we defined in 
(\ref{notref}), and also that it does not change if we define it by choosing as $\sigma$
the map contracting $\te_2$ instead of $\te_1$.

In the sequel, we shall sometimes simplify the notation and omit to indicate the map $\sigma^*$, viewing (\ref{notref})
as an inclusion.
\end{remark}

\subsection{Cut vertices}
 \label{cutsubs}
Let $G$ be a graph with a cut vertex, $v$.
Then we can write $G = H_1\vee H_2$ where $H_1$ and $H_2$ are connected subgraphs of $G$ such that $V(H_1)\cap V(H _2)= \{ v\}$ and $E(H_1)\cap E(H _2)= \emptyset$.
We say that $G = H_1\vee H_2$ is a decomposition associated to $v$.
Pick  $D_j\in \Div(H_j)$ for $j=1,2$, then  we   define     $D_1+D_2\in\Div G$  as follows\begin{displaymath}
 (D_1+D_2)(u)=\left\{ \begin{array}{ll}
D_1(v)+D_2(v) &\text{ if } u=v\\
D_1(u) &\text{ if } u\in V(H_1)-\{v\}\\
D_2(u) &\text{ if } u\in V(H_2)-\{v\}.\\
\end{array}\right.
\end{displaymath}

\begin{lemma}
\label{sepvertex}
Let $G$ be a   graph with a cut vertex and let $G = H_1\vee H_2$
be a corresponding decomposition (as described above). Let   $j=1,2$.
\begin{enumerate}
\item
\label{sepvertex1}
The map  below  is a surjective homomorphism with kernel isomorphic to $\Z$
\begin{equation}
\label{summap}
\Div (H_1)\oplus \Div (H_2) \la \Div (G); \quad \quad (D_1,D_2)\mapsto D_1+D_2 
\end{equation}
 and it induces an isomorphism
$ 
\Prin(H_1)\oplus \Prin(H_2)\cong \Prin (\G) 
$ 
and an exact sequence
$$
0\la \Z\la \Jac (H_1)\oplus \Jac (H_2)\la \Jac (G)\la 0.
$$
\item
\label{sepvertex2}  We have a commutative diagram with injective vertical arrows 
$$
\xymatrix{
0 \ar[r] & \Prin(G) \ar[r] & \Div(G) \ar[r]
& \Jac(G) \ar[r] & 0 \\
0 \ar[r]  & \Prin(H_j) \ar[r] \ar@{^{(}->}[u]&  \Div(H_j) \ar[r]\ar@{^{(}->}[u] &  \Jac(H_j) \ar[r] \ar@{^{(}->}[u] & 0 \\
}
$$

\item
\label{sepvertex3}
For every $D_1,D_2$ with $D_j\in \Div (H_j)$, we have
$$
r_{\G}(D_1+D_2)\geq \min \{r_{H_1}(D_1),r_{H_2}(D_2)\}.
$$
\item
\label{sepvertex4} 
For every $D_j\in \Div(H_j)$, we have $r_{H_j}(D_j)\geq r_G(D_j)$.
\end{enumerate}
\end{lemma}
\begin{proof}
Denote $V(H_j)=\{u_1^j,\ldots, u_{n_j}^j, v\}$ and
$V(G)=\{u_1^1,\ldots, u_{n_1}^1, v,u_1^2,\ldots, u_{n_2}^2\}$.

\medskip

(\ref{sepvertex1}). An equivalent way of defining the divisor $D_1+D_2$ is to use the two  maps
$\iota_*^j:\Div (H_j)\to  \Div (G)$ defined in (\ref{notres}). Then we have   $D_1+D_2=\iota_*^1D_1+\iota_*^2D_2$.
With this description,
it is clear that the map in part (\ref{sepvertex1}) is a surjective homomorphism. In addition, the kernel of this
 map has generator $(v , - v) \in \Div(H_1)\oplus \Div(H_2)$  and is thus isomorphic to $\mathbb Z$.

To distinguish the generators of $\Prin(H_j)$ from those of $\Prin (\G)$
we
denote by $T^j_w\in \Prin(H_j)$ the generator corresponding to $w\in V(H_j)$.
We clearly have
$$
\iota_*^jT^j_{u_h^j}=T_{u_h^j}
$$
 for $j=1,2$ and $h=1,\ldots,n_i$.
 As $\Prin(H_j)$ is freely generated by $T^j_{u_1 ^j},\ldots, T^j_{u_{n_j}^j}$  and
 $\Prin (\G)$ is freely generated by $T_{u_1 ^1},\ldots, T_{u_{n_1}^1},T_{u_1 ^2},\ldots, T_{u_{n_2}^2}$, the first part is proved.

Part (\ref{sepvertex2}) also follows from the previous argument.

(\ref{sepvertex3}). Set $r_j=r_{H_j}(D_j)$  and assume $r_1\leq r_2$. Set $D=D_1+D_2$;
to prove that $r_{\G}(D)\geq r_1$ we must show that  for every $E\in \Div_+^{r_1}(\G)$
there exists $T\in \Prin (\G)$ such that $D-E+T\geq 0$.
Pick such an $E$;
let $E_1=E_{H_1}$ and $E_2=E-E_1$, so that $E_2\in \Div H_2$.
Since $\deg E_j\leq r_1\leq r_j$ we have that there exists $T_j\in \Prin(H_j)$ such that
$D_j-E_j+T_j\geq 0$ in $H_j$. By the previous part $T=T_1+T_2\in \Prin (\G)$; let us conclude by showing that
$D-E+T\geq 0$. In fact
$$
D-E+T=D_1+D_2-E_1-E_2+T_1+T_2=(D_1-E_1+T_1)+(D_2-E_2+T_2)\geq 0.
$$

(\ref{sepvertex4}). Assume $j=1$ and set $r=r_G(D_1)$. 
By (\ref{sepvertex2}) we are free to view  $\Div(H_1)$ as a subset of $\Div(G)$.
Pick $E\in \Div^r_+(H_1)$, then
there exists $T\in \Prin(G)$ such that in $G$ we have $D_1-E+T\geq 0$.
By \eqref{sepvertex1} we know that $T=T_1+T_2$ with $T_i\in \Prin(G_i)$; since $D_1(u_h^2)=E(u_h^2)=0$
for all $h=1,\ldots, n_2$ we have that $T_2=0$, hence $D_1-E+T_1\geq 0$ in $ H_1$
 \end{proof}
Now let $G=H_1\vee H_2$ as above and let  $m,n$ be two nonnegative integers; we denote by $G^{(m,n)}$ the graph obtained by inserting
$m$ vertices in the interior of every edge of $H_1$ and $n$ vertices in the interior of every edge of $H_2$. Hence we  can write
  $G^{(m,n)} :=  H_1^{(m)} \vee H_2^{(n)}$ (recall that $H^{(m)}$ denotes the   $m$-subdivision   of a graph $H$).
 We denote by $\sigma^*_{m,n}:\Div(G) \to \Div(G^{(m,n)})$ 
  the natural map.
\begin{prop}
\label{sepref}
Let $G$ be a   graph with a cut vertex and   $G = H_1\vee H_2$
  a corresponding decomposition. Let $m,n$ be non-negative integers and 
  $G^{(m,n)}=H_1^{(m)} \vee H_2^{(n)}$ the corresponding refinement.
  Then  
\begin{enumerate}
\item
\label{sepref1}
 $\sigma^*_{m,n}(\Prin(G))\subset \Prin(G^{(m,n)})$. 
   \item
\label{sepref2}
Assume that $G$ has no loops. Then for every $D\in \Div(G)$, we have 
$$r_G(D) = r_{G^{(m,n)}}(\sigma_{m,n}^*D).$$ \end{enumerate}
\end{prop}
\begin{proof}
 It is clear that it suffices to prove part (\ref{sepref1})  for  $(0,n)$ and $(0,m)$
separately, hence it suffices to prove it for $(0,m)$.
Consider the map (for simplicity we write $\sigma^*=\sigma_{0,m}^*$)
$$ 
\sigma^*:\Div(G)=\Div (H_1 \vee H_2)\to \Div (H_1 \vee H_2^{(m)})=\Div(G^{(0,m)}).
$$ 
The group $\Prin(G)$ is generated by $\{T_u,\  \forall u\in V(G)\smallsetminus \{v\}\}$
(see Remark~\ref{oneless}). Hence it is enough to  prove that  $\sigma^*(T_u)$
is principal for all $u\in  V(G)\smallsetminus \{v\}$.
We denote by $\widehat{u}\in V(G^{(0,m)})$ the vertex corresponding to $u\in V(G)$
via the inclusion $V(G)\subset V(G^{(0,m)})$.

If $u\in V(H_1)\smallsetminus \{v\}$ we clearly  have $\sigma^*(T_u)= T_{\widehat{u}}$,
hence  $\sigma^*(T_u)\in \Prin(G^{(0,m)})$ .

Let $u\in V(H_2)\smallsetminus \{v\}$.  Denote by  $E_u(G)$   the set of edges of $G$ adjacent to $u$ and pick $e\in  E_u(G)$;
as $G^{(0,m)}$ is given by adding $m$ vertices in every edge of $G$, we will   denote the  vertices added in the interior of $e$ by 
$$
\{w_1^{e},\ldots, w_m^{e}\}\subset V(G^{(0,m)}), 
$$ 
ordering $w_1^{e},\ldots, w_m^{e}$ according to the orientation of $e$ which has $u$ as target, so that in $G^{(0,m)}$ we have $(w_m^{e}\cdot \widehat{u})=1$ and $(w_i^{e}\cdot \widehat{u})=0$
if $i<m$ (and  $(w_i^e\cdot w_{i+1}^e)=1$ for all   $i$).
One then easily checks that
$$
\sigma^*(T_u)= (m+1)T_{\widehat{u}}+\sum_{e\in E_u(G)} \sum_{i=1}^miT_{w_i^e};
$$
hence
  $\sigma^*(T_u)\in \Prin(G^{(0,m)})$, and part (\ref{sepref1}) is proved.
  
  Part (\ref{sepref2}).
  First we note that the statement holds in the case $m=n$. Indeed, in this case $G^{(n,n)}=G^{(n)}$   and hence our statement is 
 \cite[Cor. 22]{HKN}; see also \cite[Thm 1.3]{luo}. 

Using this fact, we claim that it will be enough to show only the inequality 
\begin{equation}
\label{ineq*}
r_G(D) \leq  r_{G^{(m,n)}}(\sigma_{m,n}^*D).
\end{equation}
Indeed, suppose this inequality holds for every divisor  $D$ on every graph 
of the form $G=H_1\vee H_2$ and for all pairs of integers $(m,n)$. Pick a divisor $D\in \Div(G)$, 
we   get, 
 omitting the maps $\sigma_{\ldots}^*$ for simplicity (which creates no ambiguity, as the subscript of $r$ already indicates  in which graph  we are computing the rank)
$$
r_G(D)\leq r_{G^{(m,n)}} (D) \leq  
r_{(G^{(m,n)})^{(n,m)}} (D)=
r_{G^{(l,l)}} (D)=
r_G(D) 
$$
where $l=m+n+mn$. (We used the trivial fact that for any graph $H$ and positive integers $h,k$ we have $(H^{(h)})^{(k)}= H^{(h+k+hk)}$).
Hence all the inequalities above must be equalities and the result follows.

Thus, we are left to prove Inequality~\eqref{ineq*}. Let $r=r_{G}(D)$. We have to show that for any effective divisor $E^*$ on $G^{(m,n)}$ of degree $r$
we have
$$
r_{G^{(m,n)}}( \sigma_{m,n}^*D - E^*)\geq 0.
$$
By    \cite[Thm. 1.5]{luo} (or \cite{HKN}), $V(G)$ is a rank-determining set in $G^{(m,n)}$. Therefore it will be enough to show the above claim for divisors of the form  $E^*=\sigma^*_{m,n}E$ for any effective divisor $E$ of degree $r$ on $G$. 
Summarizing, we need to show that for every   $E\in \Div ^r_+(G)$   there exists 
$T\in \Prin (G^{(m,n)})$
   such that 
\begin{equation}
\label{summ}
T + \sigma_{m,n}^*D - \sigma_{m,n}^*E \geq 0.
\end{equation}
Now, since $r=r_G(D)$, there exists  a principal divisor
$\widetilde{T}\in \Prin(G)$ such that
$$
\widetilde{T} + D-E \geq 0.
$$ 
 By 
 the previous part,  $\sigma_{m,n}^*\widetilde{T}$ is a principal divisor of $G^{(m,n)}$;
 set $T:=\sigma_{m,n}^*\widetilde{T}$.
 Then we have
 $$
0\leq   \sigma_{m,n}^*(\widetilde{T} + D-E)=T+  \sigma_{m,n}^*D - \sigma_{m,n}^*E.
$$
Therefore (\ref{summ}) holds, and we are done.
\end{proof}

 \section{Riemann-Roch for weighted  graphs}
  \subsection{Divisor theory for graphs with loops}
 Our goal here is to set up a divisor theory for graphs with loops, so that the Riemann-Roch theorem 
 holds. The Riemann-Roch theorem  has been proved for loopless graphs  
 in
  \cite{BN};
  to generalize it   we shall give a more subtle  definition for the rank 
  and for the canonical divisor.
   
 \begin{defi}
\label{loopless}
Let $\G$ be a graph and let 
  $\{ e_1,\ldots, e_c\} \subset E(G)$ be the set of its loop-edges.
We denote by
 $\hG$ the graph obtained by inserting one  vertex 
 in the interior of  the loop-edge $e_j$, for all $j=1,\ldots,c$.
Since  $V(G)\subset V(\hG)$,  
we have a canonical injective morphism 
\begin{equation}
\label{sigman}
\sigma^*:\Div(\G) \stackrel{}{\la}  \Div(\hG) .
\end{equation}
 We set
\begin{equation}
\label{defr}
\rs_{\G}(D):=r_{\hG}(\sigma^*D),
\end{equation}
and refer to $\rs_{\G}(D)$ as the {\it rank} of $D$.
\end{defi}

The superscript ``$\#$"  is used to avoid confusion with the   definition  which
disregard the loops. We often abuse notation and 
write just $r_{\hG}(D)$ omitting $\sigma^*$.

Observe that $\hG$ is free from loops and has the same genus as $G$.  (Recall that the genus of a 
connected graph $G = (V,E)$ is by definition equal to $|E|-|V|+1$.) 
With the above notation, let 
 $u_j\in V(\hG)$ be the vertex added  in the interior of $e_j$ for all $j=1,\ldots,c$.
It is clear that the map \eqref{sigman}
induces an isomorphism of $\Div(G)$ with 
 the subgroup of divisors $\hD$  on $\hG$ such that $\hD(u_j)=0$ for all  $j=1,\ldots,c$.

\begin{example}
Here is an example in the case $c=1$.
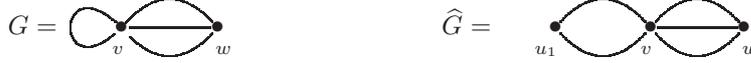
\begin{figure}[h]
\begin{equation*}
\xymatrix@=.5pc{
&&&&&&&&&&&&&&\\
G = &&\ar@{-}@(ul,dl)*{\bullet}\ar @{-}[rrr]\ar @{-} @/_.9pc/[rrr]_(.1)v_(.9)w  \ar@{-} @/^.9pc/[rrr]
&&& *{\bullet} &&&&&&&
{\widehat{G}} = &&*{\bullet} \ar @{-} @/_.9pc/[rrr]_(.1){u_1} \ar@{-} @/^.9pc/[rrr]
&&& *{\bullet}\ar @{-} @/_.9pc/[rrr]_(.1)v_(.9)w \ar@{-} @/^.9pc/[rrr]\ar @{-}[rrr]
&&& *{\bullet} &&&\\
&&&&&&&&&&&&&&\\
}
\end{equation*}
\end{figure}
\end{example}
 \begin{remark}
We have 
\begin{equation}
\label{BI}
r_G(D)\geq \rs_G(D).
\end{equation}
Indeed, let $G_0$ be the graph obtained from $G$ by removing 
all its loop-edges; then, 
by definition, $r_G(D)=r_{G_0}(D)$. On the other hand, by Lemma~\ref{sepvertex} \eqref{sepvertex4},
writing $\hG=G_0\vee H$ for some graph $H$, we have
$r_{G_0}(D) \geq r_{\hG}(D)=\rs_G(D)$, hence \eqref{BI} follows.
\end{remark}

Definition~\ref{loopless} may seem a bit arbitrary, as the choice of the refinement $\hG$ may seem arbitrary. In fact, it is natural to ask whether adding some (positive) number of vertices, different from one,  in the interior of the loop-edges of $G$ can result in a different rank. This turns out not to be the case, as we now show.

 \begin{prop}\label{thm:refinement}
Let $G$ be a graph and
 let $e_1,\ldots, e_c$ be its loop-edges.
For every $\mn=(n_1,\ldots, n_c)\in \N^c$
  let $G^{(\mn)}$  be the refinement of $G$ obtained by inserting $n_i$ vertices
in the interior of $e_i$.
Then for every $D\in \Div G$ we have
$$
\rs_G (D) = r_{G^{(\mn)}}(\sigma^*D) 
$$
where $\sigma^*:\Div(G)\ha \Div(G^{(\mn)})$ is the natural map.
\end{prop}

 \begin{proof} It will be enough to prove the proposition for $c=1$ since the general statement can be obtained easily by induction on the number of loop-edges of $G$. 

Let $H_1$ be the graph obtained from $G$ by removing its loop-edge, $e$, and let $v$ be the vertex of $G$ adjacent to $e$. We can thus decompose $G$ with respect to $v$:
$$
G=H_1\vee C_1
$$
where, for $m\geq 1$ we denote by $C_m$ the ``m-cycle", i.e., the 2-regular  graph of genus 1, having $m$ vertices and $m$ edges. Observe that for every $h\geq 1$ we have (recall that $C_m^{(h)}$ denotes the $h$-subdivision of $C_m$)
\begin{equation}
\label{Cmh}
C_m^{(h)}=C_{m(h+1)}.
\end{equation} 
Therefore,
with the notation of Proposition~\ref{sepref}, we have, for every $n\geq 0$,
\begin{equation}
\label{G0n}
G^{(0,n)} = H_1^{(0)}\vee C_{1}^{(n)} =H_1\vee C_{n+1}.
\end{equation} 
 For any divisor $D$ on $G$, by definition, we have 
 $$
\rs_{G}(D) =r_{G^{(0,1)}}(\sigma_{0,1}^* D).
 $$
 So we need to prove that for any $n\geq 1$, 
\begin{equation}
\label{end}
 r_{G^{(0,1)}}(\sigma_{0,1}^*D) = r_{G^{(0,n)}}(\sigma_{0,n}^*D).
\end{equation}  
 This is now a simple consequence of Proposition \ref{sepref} (\ref{sepref2}). Indeed, by applying  it to the loopless graph $G^{(0,1)}=H_1\vee C_2$ and the $n$-subdivision of $C_2$, we get,
 simplifying the notation by omitting the pull-back maps $\sigma^*_{\dots}$ ,
 $$
 r^{\;}_{G^{(0,1)}}(D)= r^{\;}_{(G^{(0,1)})^{(0,n)}}(D)=r_{H_1\vee C_2^{(n)}}(D)=r^{\;}_{H_1\vee C_{2n+2}^{\;}}(D)
 $$
by  (\ref{Cmh}).
On the other hand, applying the proposition a second time  to $G^{(0,n)}=H_1\vee C_{n+1}$ and the $1$-subdivision of $C_{n+1}$, we get   $$
  r^{\;}_{G^{(0,n)}}(D)= r^{\;}_{(G^{(0,n)})^{(0,1)}}(D)=
  r_{H_1\vee C_{n+1}^{(1)}}(D)=r^{\;}_{H_1\vee C_{2n+2}^{\;}}(D).
  $$ The last two equalities 
  prove (\ref{end}), hence the result is proved.
\end{proof}

 \begin{remark}
 \label{linrk}
 The definition of linear equivalence for divisors on a graph with loops can be taken to be the same  as in Subsection~\ref{graphprel}. Indeed,
let $D,D'\in \Div (G)$; then $D$ and $D'$ can be viewed as divisors on the graph  $G_0$   obtained from $G$ by removing 
all the loop-edges, or as  divisors on the graph  $\hG$.
 By Lemma~\ref{sepvertex} we have that
 $D$ and $D'$ are linearly equivalent on $G_0$ if and only  
if and only if they are linearly equivalent  on $\hG$.

\noindent It is thus obvious that if $D\sim D'$ for divisors in $\Div(G)$, then $\rs_G(D)=\rs_G(D')$.

\end{remark}
The canonical divisor $\Ks\in \Div(\G)$ of $\G$ is defined as follows
\begin{equation}
\label{defK}
\Ks:=\sum_{v\in V(\G)} (\val (v) -2)v.
\end{equation}
\begin{thm}
\label{RRloop} Let $\G$ be a graph with $c$ loops, and 
let $D\in \Div(\G)$.
\begin{enumerate}
 
\item
\label{RR}
\emph{(}Riemann-Roch theorem\emph{)} 
$$\rs_{\G}(D)-\rs_{\G}(\Ks-D)=\deg D -g +1.
$$
In particular, we have 
$\rs_{\G}(\Ks)=g-1$ and 
$\deg \Ks=2g-2$.
\item
\label{Riemann}
\emph{(}Riemann theorem\emph{)}  If $\deg D\geq 2g-1$ then $$\rs_{\G}(D)=\deg D-g.$$
\end{enumerate}
\end{thm}
\begin{proof}
Let $U=\{u_1,\ldots, u_c\}\subset V(\hG)$ be the set of vertices added to $G$ to define $\hG$.
The canonical divisor $K_{\hG}$ of $\hG$ is 
$$
K_{\hG}=\sum_{\hv\in V(\hG)} (\val (\hv) -2)\hv   =\sum_{\hv\in V(\hG)\smallsetminus U} (\val (\hv) -2)\hv
$$
because the vertices in $U$ are all 2-valent. On the other hand we have an identification
$V(G)=V(\hG)\smallsetminus U$ and it is clear that this identification preserves the valencies. Therefore, by definition (\ref{defK}) we have
$$
 \sigma^* \Ks= K_{\hG}.
$$
Hence, since the map (\ref{sigman}) is a degree preserving homomorphism,
$$
\rs_{\G}(D)-\rs_{\G}(\Ks-D)=r_{\hG}(\sigma^*D)-r_{\hG}(K_{\hG} -\sigma^*D))=\deg D-g+1
$$
where, in the last equalty, we applied the  the Riemann-Roch formula for loopless graphs
(proved by Baker-Norine in \cite{BN}), 
together with the fact that    $\G$ and $\hG$ have the same genus.

Part (\ref{Riemann}) follows from the Riemann-Roch formula we just proved,  noticing that, if $\deg D\geq 2g-1$, then $\deg \Ks-D<0$ and hence $\rs_{\G}(\Ks-D)=-1$.
\end{proof}

 The next Lemma, which we will use later,  computes the rank of a divisor on
 the so called ``rose  with $g$ petals", or ``bouquet of $g$ loops" $R_g$.
\begin{lemma}
\label{rose}
Set $g\geq 1$ and $d\leq 2g$.
Let $R_g$ be the connected graph of genus $g$ having only one vertex
(and hence $g$ loop-edges).
For the unique divisor  $D\in \Div ^d(R_g)$ we have
$$
\rs_{R_g}(D) =\left\lfloor{\frac{d}{2}}\right\rfloor.
$$
\end{lemma}
\begin{proof}
Let $v$ be the unique vertex of $\G=R_g$, hence  $D=dv$.
To compute 
$\rs_{R_g}(D)$ we must use the refinement $\hG$ of $R_g$ 
defined above. 
In this case $\hG$ is the 1-subdivision of $R_g$.
So $V(\hG)=\{\hv, u_1,\ldots, u_g\}$ with each $u_i$ of valency $2$,   and 
  $\hv$  of valency $2g$.
We have $u_i\cdot v=2$ for all $i=1,\ldots, g$, and $u_i\cdot u_j=0$ for all $i\neq j$.

Let $\hD=d\hv$ be the pull-back of $D$ to $\hG$.
Set $r:=\left\lfloor{\frac{d}{2}}\right\rfloor$.
We will first   prove that $r_{\hG}(\hD)\geq  r$.
Let 
$E$ be a degree $r$ effective divisor on $\hG$; then for some $I\subset \{1,\ldots,g\}$ we have
$$E=e_0\hv+\sum_{i\in I} e_i u_i
$$ with $e_i> 0$ and $\sum_{i=0}^r e_i=r$. Notice that $|I|\leq r$.
Now,
$$
\hD-E\sim d\hv-e_0\hv-\sum_{i\in I} e_i u_i-\sum_{i\in I}\left\lceil{\frac{e_i}{2}}\right\rceil T_{u_i}=:F.
$$
Let us prove that $F\geq 0$. Recall that $T_{u_i}(\hv)=2$, hence
$$
F(\hv)=d-e_0-2\sum_{i\in I} \left\lceil{\frac{e_i}{2}}\right\rceil  \geq d-e_0-\sum_{i\in I}(e_i+1)\geq 
2r-r
 - |I| = r-|I|\geq 0
$$
as, of course, $|I|\leq r$.
Next,   since  $T_{u_i}(u_i)=-2$ and $T_{u_i}(u_j)=0$ if $i\neq j$,  we have for all $i\in I$,
$$
F(u_i)=-e_i+2\left\lceil{\frac{e_i}{2}}\right\rceil\geq 0,
$$
and $F(u_j)=0$ for all $u_j\not\in I$.
Therefore $r_{\hG}(\hD)\geq r$. 

Finally, since  $d\leq 2g$, we can apply  Clifford's theorem
\cite[Cor. 3.5]{BN}, and therefore equality must hold.
\end{proof}

\subsection{Divisors on weighted graphs.}
\label{wgsec}
Let $(\G, \omega )$ be a {\it weighted graph}, by which we mean that
$G$ is an ordinary graph  and 
$\omega :V(G)\to \Z_{\geq 0}$ a {\it weight} function
on the vertices.
The genus, $g(G, \omega )$, of $(G,\omega)$    is  
\begin{equation}
\label{gw}
g(G, \omega )=b_1(G)+\sum_{v\in V(G)}\omega (v).
\end{equation}

We associate to $(\G, \omega )$ a weightless graph $\G^{\omega} $ as follows:
$\G^{\omega} $ is obtained by attaching at every vertex $v$ of $G$,\ 
$\omega (v)$ loops
(or ``1-cycles"), denoted by $C_v^1,\ldots ,C^{\omega (v)}_v.$
 
 We call $\G^{\omega} $ the {\it virtual} (weightless) graph of $(\G,\omega )$, and we  say that
  the $C^i_v$ are the  virtual loops. 
The initial graph $\G$ is a subgraph of $\G^{\omega} $ and 
we have an identification
\begin{equation}
\label{vertid}
V(\G)=V(\G^\omega ).
\end{equation}
It is easy to check that
\begin{equation}
\label{wg}
g (\G, \omega )=g(\G^\omega ).
\end{equation}
For the group of divisors of the weighted graph $(\G, \omega )$, we have
\begin{equation}
\label{wdiv}
\Div (\G, \omega )=\Div (G^\omega )=\Div(G). 
\end{equation}
The canonical divisor of $(\G, \omega )$ is defined as  the canonical divisor of $\G^{\omega} $, introduced in the previous section, namely,
\begin{equation}
\label{wcan}
K_{(\G, \omega )}:= K^{\#}_{\G^{\omega} }=\sum_{v\in V(G^\omega )} (\val_{G^{\omega} }(v) -2)v.
\end{equation}
Note that $K_{(\G, \omega )}\in \Div(\G,\omega )$.
By (\ref{wg}) and Theorem~\ref{RRloop} we have
$$
\deg K_{(\G, \omega )} =2g (\G, \omega )-2.
$$
For any $D\in \Div(\G, \omega )$ we define (cf. Definition~\ref{loopless})
\begin{equation}
\label{defrw} 
r_{(\G, \omega )}(D):=\rs_{\G^{\omega} }(D)=r_{\hGw}(D). 
\end{equation}

\begin{thm}
\label{wRR}
Let $(G,\omega)$ be a weighted graph.
\begin{enumerate}
\item
For every $D\in \Div(\G, \omega)$ we have
$$
r_{(\G, \omega )}(D)-r_{(\G, \omega )}(K_{(\G, \omega )}-D)=\deg D-g+1.
$$
\item
\label{wRR2}
For every $D,D'\in \Div(G)$ such that $D\sim D'$,
we have $r_{(\G, \omega )}(D)=r_{(\G, \omega )}(D')$.
\end{enumerate}
\end{thm}
\begin{proof}
The first part 
 is an immediate consequence of Theorem~\ref{RRloop}.
  
  For \eqref{wRR2}, recall Remark~\ref{linrk}; we have that $D\sim D'$ on $G$ if and only if $D$ and $D'$ are equivalent on the graph $G_0$ obtained by removing all loop-edges from $G$.
  Now, $G_0$ is a subgraph of $\hGw$, moreover $\hGw$ is obtained from $G_0$
  by attaching  a finite set of 2-cycles at some vertices of $G_0$.
 Therefore, by iterated applications of Lemma~\ref{sepvertex}, we have that $D$ is linearly equivalent to $D'$ on $\hGw$. Hence the statement follows from the fact that $r_{\hGw}$ is constant on linear equivalence classes of $\hGw$.
\end{proof}
 
\section{Specialization Lemma for weighted graphs}
\label{secspe}
In this section we fix an algebraically closed field   and assume that all schemes are of finite type over it. By ``point" we mean closed point.

By {\it nodal curve} we mean a connected, reduced, projective, one-dimensional scheme, having at most nodes (ordinary double points) as singularities.
All curves we shall consider in this section are nodal.

Let $X$ be a nodal curve; its {\it dual graph}, denoted by $G_X$, is such that $V(G_X)$ is identified with the set of irreducible components of $X$,  $E(G_X)$ is identified with the set of nodes of $X$, and there is an edge joining two vertices for every node lying at the intersection of the two corresponding components. In particular, the loop-edges of $G_X$ correspond to the nodes of the irreducible components of $X$.

The {\it weighted dual graph} of $X$, denoted by $(G_X,\omega_X)$, 
has $G_X$ as defined above, and the weight function $\omega_X$ is such that
$\omega_X(v)$ is the geometric genus of the component of $X$ corresponding to $v$.
In particular, let $g_X$ be the (arithmetic) genus of $X$, then
$$
g_X=b_1(G_X)+\sum_{v\in V(G_X)}\omega_X(v).
$$

\subsection{Specialization of families  of line bundles on curves}
 Let $\phi:\X\to B$ be a family of curves, and denote by
 $\pi: \Picphi\to B$ 
its  Picard scheme (often denoted by  $\Pic_{\X/B}$).
The set of  sections of $\pi$  is denoted as follows
  $$
  \Picphi (B):=\{\L:B\to \Picphi:\quad  \pi\circ\L=id_B \}.
  $$
  (The notation $\L$ indicates that $\L(b)$ is a line bundle on $X_b=\phi^{-1}(b)$ for every $b\in B$.) 
Let $b_0\in B$ be a closed point and  set $X_0=\phi^{-1}(b_0)$;
 denote by $(\G, \omega )$ the weighted dual graph of $X_0$.
We 
identify $\Div(\G)=\Z^{V(\G)}$, so that we have a map
\begin{equation}
\label{mdeg}
\Pic (X_0) \la \Div(\G)=\Z^{V(\G)};\quad \quad L\mapsto \mdeg \ L 
\end{equation}
where $\mdeg$ denotes the multidegree, i.e.,  for $v\in V(G)$  
the $v$-coordinate of $ \mdeg \ L$ is the degree of $L$ restricted to $v$  (recall that $V(\G)$ is identified with the set of irreducible components of $X_0$).
Finally, we have a {\it specialization} map $\tau$
\begin{equation}
\label{tau}
\Picphi(B) \stackrel{\tau}{\la}  \Div(\G) ;\quad \quad L \mapsto  \mdeg \ \L(b_0).
 \end{equation}

\begin{defi}
Let $X_0$ be a nodal  curve.
A  projective morphism $\phi:\X\to B$ of schemes
 is a {\it regular one-parameter smoothing of} $X_0$
if:
\begin{enumerate}
\item
 $B$ is   smooth,  quasi-projective, $\dim B=1;$
 \item $\X$ is a regular surface;
 \item there  is a closed point
 $b_0\in B$ such that
$X_0\cong\phi^{-1}(b_0)$. (We shall usually identify $X_0=\phi^{-1}(b_0)$.)
 \end{enumerate}\end{defi}
 
 \begin{remark}
 \label{connect2}
As we mentioned in Remark~\ref{connect},
there is a  connection between the divisor theory of $X_0$ and that of its dual graph $\G$.
 We already observed in \eqref{mdeg} that to every divisor,  or   line bundle,  on $X_0$ there is an associated
 divisor on $\G$.
 Now we need to identify $\Prin (\G)$.
 As we already said, the elements of $\Prin (\G)$
 are the multidegrees of certain divisors on $X_0$, called twisters. More precisely,
 fix $\phi:\X\to B$ a regular one-parameter smoothing of $X_0$;
 we have the following subgroup of $\Pic X_0$:
 $$
\Tw_{\phi} (X_0):=\{L\in \Pic X_0: \  
L\cong \O_{\X}(D)_{|X_0} \text{ for some }
  D\in \Div \X: \  \supp D\subset X_0 \}.
$$
The set of twisters, $\Tw(X_0)$, is defined as the union of the $\Tw_{\phi} (X_0)$
for all one-parameter smoothings $\phi$ of $X_0$.

The group $\Tw_{\phi} (X_0)$ depends on $\phi$, but its image under the multidegree map \eqref{mdeg} does not, so that $\mdeg(\Tw_{\phi} (X_0))=\mdeg(\Tw (X_0))$.
Moreover, the multidegree map
induces an identification between
the multidegrees of all twisters and $\Prin (\G)$:
$$
\mdeg (\Tw (X_0))=\Prin (\G)\subset \Z^{V(\G)}.
$$
See \cite{cner}, \cite [Lemma 2.1]{bakersp} or \cite{CBNgraph} for details.
\end{remark}

\begin{defi}
\label{phieqdef}
 Let $\phi$ be a regular one-parameter smoothing of  $X_0$ and let 
 $\L,\L'\in \Picphi(B)$.
 We define $\L$ and $\L'$ to be  $\phi$-equivalent,   writing $ \L \sim _{\phi}\L'$, as follows
\begin{equation}
\label{phieq}
 \L \sim _{\phi}\L'\ \  \   \text{ if } \  \  \L(b)\cong \L'(b),\  \  \forall b\neq b_0.
\end{equation}
 \end{defi}

 \begin{example}
 \label{L(C)} 
Let $\phi$ be as in the definition and let $C\subset X_0$ be an irreducible component. Denote by $\L'=\L(C)\in \Picphi(B)$
the section of $\Picphi \to B$ defined as follows:
$\L'(b)=\L(b)$  if $b\neq b_0$ and $\L'(b_0)=\L\otimes \O_{\X}(C)\otimes \O_{X_0}$.
Then $\L(C)\sim _{\phi}\L$.
The same holds replacing $C$ with any  $\Z$-linear combination   of the components of $X_0$.
\end{example}

 \begin{lemma}
 \label{philm} Let $\phi$ be a regular one-parameter smoothing of  $X_0$ and let 
 $\L,\L'\in \Picphi(B)$ such that $\L\sim _{\phi}\L'$.
 Then the following hold.
\begin{enumerate}
\item
 \label{philm1}
$\tau(\L)\sim \tau (\L')$.
\item
 \label{philm2}
If  $h^0(X_b,\L(b))\geq r+1$ for every $b\in B\smallsetminus {b_0}$,
then   $h^0(X_b,\L'(b))\geq r+1$
for every $b\in B$.
\end{enumerate}
\end{lemma}
\begin{proof}
To prove both parts we can replace $\phi$ by a finite \'etale base change
(see \cite[Claim 4.6]{CBNgraph}). Hence we can assume that $\L$ and $\L'$ are given by line bundles
on $\X$, denoted again by $\L$ and $\L'$.

\eqref{philm1}.
Since $\L$ and $\L'$ coincide on every fiber but the special one,
there exists a divisor $D\in \Div \X$ such that $ \supp D\subset X_0$ for which
$$
\L\cong\L'\otimes \O_{\X}(D).
$$
Using Remark~\ref{connect2} we have
$\O_{\X}(D)_{|X_0}\in \Tw (X_0)$ and 
$$
\tau  (\O_{\X}(D))=\mdeg \ \O_{\X}(D)_{|X_0}\in \Prin (\G)
$$
so we are done.

\eqref{philm2}. This is a straightforward consequence of the upper-semicontinuity of $h^0$.
\end{proof}
By the Lemma, we have a commutative diagram:
 \begin{equation}\label{diag1}
\xymatrix{
\Picphi(B)\ar@{^{}->}[r]^{\tau}  \ar@{->>}[d]_{} & \Div(G)\ar@{->>}[d] \\
\Picphi(B)/_{\sim_{\phi}}\ar@{->}[r] &\Jac(G)
}
\end{equation}
and, by Remark~\ref{connect2}, the image of $\tau$ contains $\Prin(G)$.

\subsection{Weighted Specialization Lemma} 
We shall now prove Theorem~\ref{spe},
  generalizing   the original specialization Lemma \cite[Lemma 2.8]{bakersp} to
  weighted graphs. Our set-up is similar to that of \cite[Prop.4.4]{CBNgraph}, which is   Theorem~\ref{spe}
 for the (easy) special  case of   
 weightless graphs admitting loops.     
  Before proving  Theorem~\ref{spe} we need some preliminaries.
  
  \medskip
  
  Let $\G$ be a connected graph.  
For $v,u\in V(\G)$, denote by
$ d(v,u) $ the distance between $u$ and $v$ in $G$; note that $ d(v,u) $ is the minimum  length of a path  joining $v$ with $u$,
so that $d(v,u)\in \Z_{\geq 0}$ and  $d(v,u)=0$ if and only if $v=u$.

\medskip

Fix $v_0\in V(\G)$;
we now define an ordered partition of $V(\G)$ (associated to $v_0$) by looking at the distances to $v_0$.
For $i\in \Z_{\geq 0}$ set
$$
Z_i^{(v_0)}:=\{u\in V(\G): d(v_0,u)=i\};
$$
we have $Z_0^{(v_0)}=\{v_0\}$ and, obviously, there exists an $m$ such that 
 $Z_n^{(v_0)}\neq \emptyset $ if and only if $0\leq n\leq m$.
We have thus an ordered partition  of $V(\G)$
\begin{equation}
\label{partition}
V(\G)=Z_0^{(v_0)}\sqcup\ldots \sqcup Z_m^{(v_0)}.
\end{equation}
 We refer to it as {\it the distance-based partition starting  at $v_0$}.
 We will  often   omit the superscript $(v_0)$.

\begin{remark}
One checks easily that 
for every $u\in V(\G)\smallsetminus \{v_0\}$ with $u\in Z_i$  and for any $0\leq i\neq j\leq m$,  we have
\begin{equation}
\label{partitionw}
u \cdot Z_j\neq 0 \  \  \text{ if and only if } \  \  j=i\pm 1.
\end{equation}
Therefore for any $0\leq i\neq j\leq m$, we have
$
Z_i\cdot Z_j\neq 0$  if and only if $|i-j|=1.$

\end{remark}
Whenever $\G$ is the dual graph of a curve $X_0$, we identify $V(\G)$ with the components of $X_0$ without further mention and with no change in notation. Similarly, a subset of vertices $Z\subset V(\G)$
 determines a subcurve of $X_0$ (the subcurve whose components
 are the vertices in $Z$)
  which we denote again by $Z$.

 \medskip

The following result will be used to prove Theorem~\ref{spe}.

\begin{prop}
\label{component-general}
Let $X_0$ be a nodal curve,  $C_0, C_1,\dots, C_n \subset X_0$  its irreducible components 
  of arithmetic genera $g_0, g_1, \dots, g_n$, respectively, 
and $\G$ the dual graph of $X_0$.
Fix  $\phi:\X \to B$ a regular one-parameter 
smoothing of    $X_0$,     and
 $\L\in \Picphi(B)$ such that
$ 
h^0(X_b, \L (b))\geq  r+1>0
$ for every $b\in B$. Consider a sequence $r_0, r_1,\dots, r_n$ of non-negative integers such that $r_0 + r_1 +\dots +r_n =r$. Then there exists 
an effective divisor $E\in \Div(\G)$ such that
$E\sim \tau(\L)$ 
and for any $0\leq i\leq n$
\begin{equation}\label{eq1}
E(C_i)\geq\left\{ \begin{array}{ll}
2r_i & \text { if } r_i\leq g_i-1\\
\\
 r_i+g_i & \text { if } r_i\geq g_i\\
 \end{array}\right.
\end{equation}
 (viewing $C_i$ as a vertex of $G$, as usual).
 \end{prop}

In the proof we are going to repeatedly use the following easy observation.
  \begin{claim}\label{claim:increasing}
 Let $g$ be a nonnegative integer and  $s: \mathbb N \rightarrow \mathbb N$  the function defined by  
$$
s(t)=\left\{ \begin{array}{ll}
2t & \text { if } t\leq g-1\\
\\
t+g& \text { if } t\geq g.\\
 \end{array}\right.
$$
 \begin{enumerate}
 \item
 \label{claim:increasing1}
$s(t)$ is an increasing function. 
\item 
\label{claim:increasing2}
 Let $C$ be an irreducible nodal curve of genus $g$ and $M$ a line bundle of degree $s(t)$ on $C$.
 Then $h^0(C,M)\leq t+1$.
 \end{enumerate}
\end{claim}
\begin{proof}
Part \eqref{claim:increasing1} is trivial. Part \eqref{claim:increasing2} is an immediate consequence of Clifford's inequality and Riemann's theorem
(which are well known to hold on an irreducible nodal curve $C$). 
\end{proof}
\begin{proof} [\it Proof of Proposition~\ref{component-general}]
Consider the distance-based partition \  $V(\G)=Z_0\sqcup\ldots \sqcup Z_m$ starting at $C_0$, defined in (\ref{partition}).
For every $i$ the vertex set $Z_i$ corresponds  to a subcurve, also written $Z_i$, of $X_0$.  We thus get a decomposition
$X_0=Z_0\cup\ldots \cup Z_m$.

We denote by $s_i$ the quantity appearing in the right term of inequalities~\eqref{eq1}:
$s_i:= 2r_i$ if $r_i \leq g_i-1$ and $s_i = r_i+g_i$ if $r_i \geq g_i$.

The proof of the proposition proceeds by an induction on $r$.

For the base of the induction, i.e.  the case $r=0$, 
we have  $r_i =0$ for all $i\geq 0$. 
We have to show the existence of 
an effective divisor $E\in \Div(\G)$ such that
$E\sim \tau(\L)$. This trivially 
follows from our hypothesis because $\mathcal{L}(b_0)$ 
has a nonzero global section and so $\tau(\mathcal{L})$ itself is effective.

\

Consider now $r \geq 1$ and assume 
without loss of generality that $r_0 \neq 0$. 
By  the induction hypothesis (applied to $r-r_0$ and the sequence $r'_0=0, r_1'=r_1, \dots, r'_n=r_n$)
we can choose $\L$ so that  
  for the divisor $E =\tau(\mathcal L)$, all the Inequalities~\eqref{eq1} 
  are verified for $i\geq 1$, and $E(C_0) \geq 0$. 
  Furthermore,  we will assume that  $E$   maximizes the vector 
  $(E(C_0), E(Z_1), \dots, E(Z_m))$ in the lexicographic order, i.e., 
  $E(C_0)$ is maximum among all elements in $ |\tau(\L)|$ verifying Inequalities~\eqref{eq1} 
  for $i \geq 1$,  next, we require that $E(Z_1)$ be maximum among all  such
$E$, and so on. Up to changing $\L$ within its   $\phi$-equivalence class we can assume that
$E=\tau (\L)$.  Note that by Lemma~\ref{philm}(2), 
the new $\mathcal{L}$ is still satisfying the hypothesis of the proposition.

\noindent In order to prove the proposition, we need to show that $E(C_0) \geq s_0$. 

\medskip

We now consider  (see example~\ref{L(C)})
$$\L':=\L(-C_0)\in \Picphi(B).$$ 
We denote  $L_0=\L(b_0)\in \Pic(X_0)$, and similarly $L'_0=\L'(b_0)\in \Pic(X_0)$. 
  
\begin{claim}\label{claim2}
The dimension of the space of global sections of $L'_0$ which  identically vanish on $\overline{X_0\smallsetminus C_0} $ is at least $r_0 +1$. 
\end{claim}

Set  $W_0=\overline{X_0\smallsetminus C_0}.$ 
To prove the claim, set $E'=\tau(\L')=\mdeg L'_0$, so that $E'\sim E$.
Now,   for every component $C\subset X_0$
we have
\begin{equation}
\label{degL'}
E'(C)=\deg _{C}L'_0=E(C)-C\cdot C_0;
\end{equation}
in particular  $E'(C_0) >E(C_0)$. 
Therefore, by the maximality of $E(C_0)$, the divisor $E'_0$ does not verify 
some of the inequalities in~\eqref{eq1} for $i\geq 1$, and so the subcurve $Y_1\subset X_0$ defined below is not empty
$$
Y_1:=
\bigcup_{E'(C_i)< s_i}C_i=
 \bigcup_{E(C_i)+C_i\cdot W_0<s_i}C_i.
$$
Since the degree of $L'_0$ on each component $C_i$ of $Y_1$ is strictly smaller than $s_i$,
  by Claim~\ref{claim:increasing}\eqref{claim:increasing2} on  $C_i$ we have
$h^0(C_i,  L'_0 )\leq r_i$.
Let $\Lambda_1\subset H^0(X_0, L'_0)$ be the space of sections   which  vanish on $Y_1$, so that 
  we have  a  series of maps 
$$
0\la \Lambda_1=\ker\rho \la H^0(X_0, L'_0)\stackrel{\rho}{\la}  H^0(Y_1, L'_0)\hookrightarrow \bigoplus_{C_i \subset Y_1}H^0(C_i, L'_0)
$$\
where $\rho$ denotes the restriction. 
From this sequence and the above estimate we get 
$$
\dim \Lambda_1\geq h^0(X_0, L'_0)-\sum_{i: C_i \subset  Y_1} r_i\geq r+1-\sum_{i\geq 1} r_i=r_0+1.
$$
Hence  we are done if $Y_1=W_0$. Otherwise,    for $h\geq 2$ we iterate, setting
$$
W_{h-1}:=\overline{X_0\smallsetminus (C_0\cup Y_1\cup\ldots \cup Y_{h-1})}
 \quad \quad
{ \text{ and }}   \quad \quad
Y_h:=\bigcup_{\stackrel{C_i\subset W_{h-1},}{E(C_i) +C_i\cdot W_{h-1}<s_i}}C_i.
$$
Let $\Lambda_h\subset H^0(X_0, L'_0)$ denote  the space of sections   which  identically vanish on $Y_1 \cup \dots \cup Y_h$. 
We will prove  that  $\codim \Lambda_h\leq   \sum_{i: C_i \subset Y_1\cup \dots \cup Y_{h}}r_i$, and that $Y_h$ is empty only if $W_{h-1}$ is empty. This will finish the proof of Claim~\ref{claim2}.

To prove the first statement we use  induction on $h$. 
The base case $h=1$ has been done above.
Consider    $C_j\subset Y_h$, so that
  $E(C_j) <s_j-C_j\cdot W_{h-1}$, hence
$$
E'(C_j)=E(C_j)-C_0\cdot C_j<s_j -C_j\cdot  W_{h-1}-C_0\cdot C_j =s_j  + C_j\cdot (\sum_{i=1}^{h-1}Y_i).
$$
 as  $C_j\cdot W_{h-1}=-C_j\cdot (C_0+\sum_{i=1}^{h-1}Y_i)$.
Hence $(L'_0)_{|C_j}(- C_j\cdot\sum_{i=1}^{h-1}Y_i)$ has degree smaller than $s_j $, therefore   by Claim~\ref{claim:increasing}\eqref{claim:increasing2} on  $C_j$,
\begin{equation}
\label{Cj}
h^0(C_j, L'_0(- C_j\cdot\sum_{i=1}^{h-1}Y_i)\leq r_j.
\end{equation}
Let us denote   by $\rho_h:\Lambda_{h-1} \to H^0(Y_h, L'_0)$ the restriction map. Then
we have  the following series of maps 
$$
0\la \Lambda_h=\ker \rho_h\la \Lambda_{h-1} \stackrel{\rho_h}{\la}  {\operatorname {Im}} \rho_h \hookrightarrow \bigoplus_{C_j \subset Y_h}H^0(C_j, L'_0(- C_j\cdot\sum_{i=1}^{h-1}Y_i).$$
Hence  the codimension of $\Lambda_h$ in $\Lambda_{h-1}$, written $\codim_{\Lambda_{h-1}}\Lambda_h$,  is at most the dimension of the space on the right, which, by \eqref{Cj},  is at most
$\sum _{j:C_j \subset Y_h}r_j$. Therefore
$$
\codim \Lambda_h=\codim \Lambda_{h-1}+\codim_{\Lambda_{h-1}}\Lambda_h\leq 
\sum_{i:C_i\subset Y_1\cup \dots \cup Y_{h-1}} r_i+\sum _{j:C_j \subset Y_h}r_j
$$
where we used  the  induction hypothesis on $\Lambda_{h-1}$. The first claim is proved.

For the proof of the second statement, suppose, by contradiction, $Y_h=\emptyset$ and 
$ W_{h-1}\neq \emptyset$. Set
\begin{equation}
\label{Eh}
E_{h}:=E+T_{W_{h-1}}
\end{equation} 
 where $T_{W_{h-1}}\in \Prin(\G)$ as defined in (\ref{defT});
hence  $E_{h}\sim E$. 

Since $Y_h$ is empty, we get $E_{h}( C )\geq s_i$ for any 
$C \subseteq W_{h-1}$. On the other hand, 
for any $C \subset \overline{X\smallsetminus W_{h-1}}$, we have $E_{h}( C ) \geq E( C )$.
Therefore, by the choice of $E$, and the maximality assumption, we must have
$E_{h}( C_0) =E( C_0 )$, i.e.,  $W_{h-1}\cdot C_0=0$. 
Therefore 
$W_{h-1}\subset \cup_{j\geq 2}Z_j$ and hence
 $W_{h-1}\cdot Z_1\geq 0$. In particular, we have $E_h(Z_1) \geq E(Z_1)$. But, by
the maximality of $E(Z_1)$, we must have
$E_{h}(Z_1)=E(Z_1)$, i.e., $W_{h-1}\cdot Z_1=0$.
Therefore  
$W_{h-1}\subset \cup_{j\geq 3}Z_j$.
Repeating this argument, we conclude that  $W_{h-1}\subset  Z_{m+1} = \emptyset,$ which is a contradiction.
Claim~\ref{claim2} is proved.

\

Let $\Lambda$ be the set of sections of $L'_0$ which identically vanish on $W_0$;   by the   claim, $\dim \Lambda \geq r_0+1$.  We have a natural injection
$\Lambda \ha H^0(C_0, L'_0(-C_0\cap W_0))=H^0(C_0,L_0)$, hence
 $r_0+1 \leq h^0(C_0, L_0).$

Set $\widehat{r_0}:=h^0(C_0, L_0)-1$ so that $\widehat{r_0}\geq r_0$. 
 By Claim~\ref{claim:increasing}\eqref{claim:increasing2} on    $C_0$  we obtain,
\begin{displaymath}
E(C_0)=\deg _{C_0}L_0\  \left\{ \begin{array}{ll}
\geq 2\widehat{r_0} & \text { if } \widehat{r_0}\leq g_0-1\\
\\
= \widehat{r_0}+g_0 & \text { if } \widehat{r_0}\geq g_0.\\
\end{array}\right .
\end{displaymath}
By Claim~\ref{claim:increasing} \eqref{claim:increasing1}, we infer that $E(C_0) \geq s_0$ , and the proof of Proposition~\ref{component-general} is  complete.
\end{proof}

 \begin{thm}[Specialization Lemma]
\label{spe}
Let $\phi:\X \to B$ be a regular one-parameter 
smoothing of  a projective nodal  curve $X_0$.
Let $(\G,\omega )$ be the weighted dual graph of $X_0$.
Then for every $\L\in \Picphi(B)$ there exists an open neighborhood $U\subset B$ of $b_0$ such that for every $b\in U$  such that $b\neq b_0$
\begin{equation}
\label{mixed}
r(X_b, \L (b))\leq r_{(\G,\omega )}(\tau (\L)).
\end{equation}
\end{thm}

\begin{proof}
To simplify the presentation, we will assume   $G$   free from loops,
and indicate, at the end, the (trivial)  modifications   needed to get the proof in general.

\medskip

Up to restricting $B$ to an open neighborhood of $b_0$ we can assume that for some $r\geq -1$ and  for every $b\in B$ we have
\begin{equation}
\label{h0}
h^0(X_b, \L(b))\geq r+1  
\end{equation}
with equality for $b\neq b_0$. Set $D=\tau (\L)$;
 we must prove that $r_{(\G,\omega )}(D)\geq r$.

As in Proposition~\ref{component-general}, we write $C_0, C_1,\dots, C_n$ for the irreducible components of $X$,  
 with $C_i$ of   genus $g_i$. We denote by $v_i\in V(G)$ the vertex corresponding to $C_i$.

Recall that
we denote by  $\hGw$ 
the weightless, loopless graph   obtained from $G$ by adding $g_i=\omega(v_i)$ \  2-cycles at $v_i$
for every $v_i\in V(G)$.
We have a natural injection (viewed as an inclusion)
$ 
\Div(G)\subset    \Div (\hGw)
$  and,
 by definition, $r_{(\G,\omega )}(D)=r_{\hGw}(D)$.
Summarizing, we must prove  that
 \begin{equation}
\label{thspec}
r_{\hGw}(D)\geq r.
\end{equation}
The  specialization Lemma for weightless graphs   gives  that
the rank of  $D$, as a divisor on the  weightless  graph $\G$, satisfies
\begin{equation}
\label{oldspec}
r_{\G}(D) \geq r.
\end{equation}

Now observe that the graph obtained by removing from $\hGw$ every edge of $G$
is a disconnected (unless $n=0$) graph $R$ of type
$$
R=\sqcup_{i=0}^nR_i
$$
 where $R_i=\widehat{R_{g_i}}$ is the refinement of the ``rose" $R_{g_i}$ introduced in \ref{rose}, 
for every $i=0,\ldots,n$. Note that if $g_i=0$, the graph $R_i$ is just the  vertex $v_i$ with no edge.
 
Now, extending   the notation of \ref{sepvertex} to the case of multiple cut-vertices,
 we have  the following decomposition of $\hGw$
$$
\hGw=G\vee R 
$$
with $G\cap R=\{v_0,\ldots, v_n\}$.
By Lemma~\ref{sepvertex}\eqref{sepvertex3}   for any $D\in \Div(G)$ such that $r_{G}(D)\geq 0$ we have
  $r_{\hGw}(D)\geq 0$.

\medskip

We are ready to prove (\ref{thspec}) using induction on $r$.
If $r=-1$ there is nothing to prove. If $r=0$,   by (\ref{oldspec}) we have $r_{G}(D)\geq 0$
and hence, by what we just observed,
$r_{\hGw}(D)\geq 0$. So we are done.

Let $r\geq 1$ and pick an effective divisor $E\in \Div^r(\hGw)$.

\noindent Suppose first that $E(v)=0$ for all $v\in V(G)$;   in particular, $E$ is entirely supported on $R$.   
We write $r_i$ for the degree of the restriction of $E$ to $R_i$, so that for every $i=0,\ldots, n,$
we have 
\begin{equation}
\label{inr}
r_i\geq 0 \quad\quad \quad {\text{ and }} \quad\quad \quad\sum_{i=0}^nr_i=r.
\end{equation}
 It is clear that   it suffices to prove the existence of an effective divisor $F \sim D$ such that
the restrictions $F_{R_i}$ and $E_{R_i}$ to $R_i$ satisfy
$r_{R_i}(F_{R_i}-E_{R_i})\geq 0$ for every $i=0,\ldots, n$.
 
\noindent By Proposition~\ref{component-general} there exists an effective divisor $F \sim D$ so that
\eqref{eq1} holds  for every $i=0,\ldots, n$, i.e. 
\begin{displaymath}
F(C_i)\geq\left\{ \begin{array}{ll}
2r_i & \text { if } r_i\leq g_i-1\\
\\
 r_i+g_i & \text { if } r_i\geq g_i.\\
\end{array}\right .
\end{displaymath}
(Proposition~\ref{component-general} applies because of the relations \eqref{inr}).
Now, $F(C_i)$ equals the degree of $F_{R_i}$, hence 
by the above  estimate combined with Theorem~\ref{RRloop}(\ref{Riemann}) and Lemma~\ref{rose},
one easily checks that $r_{R_i}(F_{R_i})\geq r_i$, 
hence,   $r_{R_i}(F_{R_i}-E_{R_i})\geq 0$.

We can now assume that $E(v)\neq 0$ for some 
$v\in V(G) \subset V(\hGw)$.
We write $E=E'+v$ with $E'\geq 0$ and $\deg E'=r-1$.

Arguing as for \cite[Claim 4.6]{CBNgraph}, we are free to
 replace $\phi:\X\to B$ by a finite  \'etale base change. Therefore we can assume that $\phi$ has a section 
$\sigma$ passing through the component of $X_0$ corresponding to $v$.  It is clear that for every $b\in B$
we have
$$
r(X_b, L_b(-\sigma(b)))\geq r(X_b, L_b)-1\geq r-1.
$$
Now, the specialization of $\L\otimes \O(-\sigma(B))$ is $D-v$, i.e.,
$$ 
\tau (\L\otimes \O(-\sigma(B)))=D-v.
$$
By induction we have $r_{\hGw}(D-v)\geq r-1$.
Hence, the degree of $E'$ being $r-1$, there exists $T\in \Prin (\hGw)$ such that
$$
0\leq D-v -E' +T= D-v-(E-v)+T=D-E+T.
$$
We thus proved that $0\leq r_{\hGw}(D-E)$ for every effective $E\in \Div ^r(\hGw)$.
This proves (\ref{thspec}) and hence the
 theorem, in case $G$ has no loops.

\medskip 

If $G$ admits some loops, let  $G'\subset G$ be the graph obtained by removing from $G$ 
all of its loop edges. Then $\hGw$ is obtained from $G'$ by adding to the vertex $v_i$ exactly
 $g_i$ \  2-cycles, where $g_i$ is the arithmetic genus of $C_i$
 (note than $g_i$ is now equal to $\omega(v_i)$ plus the number of loops adjacent to $v_i$ in $G$).
Now replace $G$ by $G'$ and use exactly the same proof. (Alternatively,
one could apply the same argument used in \cite[Prop. 5.5]{CBNgraph}, where the original Specialization Lemma of \cite{bakersp} was  
extended to weightless graphs admitting loops.)
\end{proof}
\section{Riemann-Roch on    weighted tropical curves}
\subsection{Weighted tropical curves as pseudo metric graphs}
Let $\Gamma=(\G, \omega,\ell)$ be a weighted tropical curve,
that is, $(\G,\omega )$ is a weighted graph (see Section~\ref{wgsec}) and $\ell:E(\G)\to \R_{>0}$ is a (finite) length function on the edges.
We also say that $(\G, \ell)$ is  a {\it metric graph}.

If $\omega$ is the zero function, we write $\omega=\mo$ and say that the tropical curve is {\it pure}.

Weighted tropical curves were used in \cite{BMV} to bordify the space of pure tropical curves;
notice however that we use the slightly different terminology of \cite{CHBK}.

For  pure tropical curves there exists a good divisor theory for which   the  Riemann-Roch theorem holds, as proved by Gathmann-Kerber in \cite{GK} and by Mikhalkin-Zharkov in \cite{MZ}. The purpose of this section is 
 to   extend this to the weighted setting.
 
 \
 
\noindent{\it Divisor theory on pure tropical curves.}
Let us quickly recall the set-up  for pure tropical curves; we refer to \cite{GK} for details. Let $\Gamma=(\G, \mo,\ell)$ be a pure tropical curve.
The group of divisors of $\Gamma$ is the free abelian group 
$\Div(\Gamma)$ generated by the points of $\Gamma$.

 A {\it rational function} on $\Gamma$ is a continuous function $f:\Gamma \to \R$
such that the restriction of $f$ to every edge of $\Gamma$ is a piecewise affine
integral function
(i.e., piecewise of type $f(x)=ax+b$, with $a\in \Z$) 
 having finitely many pieces.

Let $p\in \Gamma$ and let $f$ be a rational function as above. The order of $f$ at $p$, written 
$\ord_pf$, is the sum of all the slopes of $f$ on the outgoing segments of $\Gamma$ adjacent to $p$.
The number of such segments is equal to the valency of $p$ if $p$ is a vertex of $\Gamma$, and is equal to 2 otherwise. The divisor of $f$ is defined as follows
$$
\div (f):=\sum_{p\in \Gamma}\ord_p(f)p\in \Div(\Gamma).
$$
Recall that $\div f$ has degree $0$. The  divisors of the from $\div (f)$ are called {\it principal} and they form a subgroup of $\Div (\Gamma)$, denoted by $\Prin (\Gamma)$. Two divisors $D,D'$ on $\Gamma$ are said to be linearly equivalent,
written $D\sim D'$,  if 
$D-D'\in \Prin (\Gamma)$.

Let $D\in \Div \Gamma$. Then $R(D)$ denotes  the set of rational functions on $\Gamma $ such that $\div (f)+D\geq 0$.
The rank of $D$ is defined as follows
$$
r_{\Gamma}(D):=\max \,\{k:\  \forall E\in \Div^k_+ (\Gamma),  \  \  R(D-E)\neq \emptyset\}
$$
so that $r_{\Gamma}(D)=-1$ if and only if $R(D)=\emptyset$.

The following trivial remark is a useful consequence of the definition.
\begin{remark}
\label{trivrk}
Let $\Gamma_1$ and $\Gamma_2$ be  pure tropical curves and let
$\psi:\Div(\Gamma_1)\to\Div(\Gamma_2)$ be a group isomorphism 
inducing an isomorphism of effective and principal divisors (i.e.,
$\psi(D)\geq 0$ if and only if $D\geq 0$, and $\psi(D)\in \Prin (\Gamma_2)$ 
 if and only if $D\in \Prin (\Gamma_1)$).
 Then for every $D\in \Div(\Gamma_1)$ we have
 $r_{\Gamma_1}(D)=r_{\Gamma_2}(\psi(D)).$
 \end{remark}

\

To extend the theory to the weighted setting, our
  starting point is to give weighted tropical curves  a geometric interpretation     by
what we call pseudo-metric graphs.
\begin{defi}
A {\it pseudo-metric graph} is a pair $(\G,\ell)$ where $\G$ is a  graph and $\ell$ a {\it pseudo-length} function
$\ell:E(\G)\to \R_{\geq 0}$ which is allowed to vanish only on loop-edges of $\G$
(that is, if $\ell(e)=0$ then $e$ is a loop-edge of $G$).
\end{defi}
Let $\Gamma=(\G, \omega,\ell)$ be a weighted tropical curve, we associate to it the pseudo-metric graph,
 $ (\G^{\omega} ,\ell^\omega )$, defined as follows.  $\G^{\omega} $ is the ``virtual" weightless graph associated to $(\G,\omega )$ described in   subsection~\ref{wgsec} ($\G^{\omega} $ is obtained by attaching to $\G$ exactly $\omega (v)$ loops based at every vertex $v$);
 the function $\ell^\omega :E(\G^\omega )\to \R_{\geq 0}$ is the extension of $\ell$ vanishing at all the virtual loops.

It is clear that $(\G^{\omega} ,\ell^\omega )$ is uniquely determined. Conversely, to any pseudometric graph
$(\G_0,\ell_0)$ 
we can associate a unique
weighted tropical curve $(\G, \omega ,\ell)$ 
such that 
$G_0=G^{\omega}$ and $\ell_0=\ell^{\omega}$
as follows. $\G$ is the subgraph of $\G_0$ obtained by removing
every loop-edge $e\in E(\G)$ such that $\ell_0(e)=0$. Next, $\ell$ is the restriction of $\ell_0$ to $\G$; finally, 
for any $v\in V(\G)=V(\G_0)$ the weight  $\omega (v)$ is defined to be equal to the number of loop-edges of $\G^0$
adjacent to $v$ and having
length $0$.

Summarizing, we have proved the following.
\begin{prop}
\label{pseudo}
The map associating to the weighted tropical curve $\Gamma=(\G, \omega,\ell)$
the pseudometric graph $(\G^{\omega} ,\ell^\omega )$ is a bijection between
the set of weighted tropical curves and the set of pseudometric graphs, extending 
 the bijection between pure tropical curves and metric graphs (see   \cite{MZ}).
\end{prop}
\subsection{Divisors on weighted tropical curves.}
Let $\Gamma=(\G, \omega,\ell)$ be a weighted tropical curve.
There is a unique pure tropical curve having the same metric graph as $\Gamma$, namely the curve
$
\Gamma^{\mo}:=(\G, \mo,\ell).
$
Exactly as for pure tropical curves, we define the group of divisors of $\Gamma$   as the free abelian group generated by the points of $\Gamma$:
$$
\Div (\Gamma)= \Div(\Gamma^{\mo})=\{\sum_{i=1}^mn_i p_i,\  n_i\in \Z, \  p_i\in (G,\ell)\}.
$$
The canonical divisor of $\Gamma$ is
$$
K_{\Gamma}: =\sum_{v\in V(G)} (\val (v) +2\omega (v)-2)v  
$$
where $\val (v)$ is the valency of $v$ as vertex of the graph $G$.
Observe that there is an obvious identification of $K_{\Gamma}$ with $K_{(G, \omega )}$, in other words,
  the canonical divisor of $K_{\Gamma}$ is the canonical divisor of the virtual graph $\G^{\omega} $ associated to $(\G,\omega )$.

Consider the pseudo-metric graph associated to $\Gamma$ by the previous proposition: 
$(\G^{\omega} ,\ell ^\omega )$. Note that $(\G^{\omega} ,\ell ^\omega )$  is not  a tropical curve  as the length function  vanishes at the virtual edges.
We then define  a pure tropical curve, $\Gamma^{\omega} _{\e}$, for every $\e>0$
$$
\Gamma^{\omega} _{\e}=(\G^{\omega} ,\mo,\ell_{\e} ^\omega )
$$ where $\ell_\e^{\omega}  (e)=\e$ for every edge lying in  some virtual cycle, and $\ell_\e^\omega (e)=\ell(e)$ otherwise.
Therefore $(\G^{\omega},\ell ^\omega )$ is the limit of $\Gamma^{\omega} _{\e}$ as $\e$ goes to zero.
Notice that 
for every curve  $\Gamma^{\omega} _{\e}$ we have a natural inclusion

$$ 
\Gamma^{\mo}\subset \Gamma^{\omega} _{\e}
$$
(with $\Gamma^{\mo}$ introduced at the beginning of the subsection).  
We refer to the loops given by $\Gamma^{\omega} _{\e}\smallsetminus \Gamma^{\mo}$
as {\it virtual} loops.

Now, we have natural injective homomorphism for every $\e$
\begin{equation}
\label{iotae}
\iota_{\epsilon}:\Div (\Gamma)\ha  \Div(\Gamma_\e^\omega ) 
\end{equation}
and it is clear that $\iota_{\epsilon}$ induces an isomorphism 
of $\Div (\Gamma)$ with the subgroup of divisors on $\Gamma_\e^\omega$ supported on
 $\Gamma^{\mo}$.
\begin{thm}
\label{RRwc}
Let $\Gamma=(G,\omega,\ell)$ be a weighted tropical curve of genus $g$ and let
  $D\in \Div (\Gamma)$.
Using the above notation, the following hold.
\begin{enumerate}
\item 
\label{RRwc1}
The number $r_{\Gamma_\e^{\omega} }(\iota_{\epsilon}(D))$ is independent of $\e$. 
Hence we define 
$$
r_\Gamma(D):=r_{\Gamma_\e^{\omega} }(\iota_{\epsilon}(D)).$$
\item
\label{RRwc2} \emph{(}Riemann-Roch\emph{)} With the above definition, we have
$$
r_{\Gamma}(D)-r_{\Gamma}(K_{\Gamma}-D)=\deg D-g+1.
$$
\end{enumerate}
\end{thm}
\begin{proof}
The proof of (\ref{RRwc1}) can be obtained by a direct limit argument to compute $r_{{\Gamma}_\e^{\omega} }(D)$, using Proposition~\ref{thm:refinement}.  A direct proof    is as follows. 

For two $\e_1, \e_2>0$, consider the 
homothety of ratio $\e_2/\e_1$ on all the virtual loops. This produces a homeomorphism$$\psi^{(\e_1,\e_2)}: \Gamma^\omega_{\e_1} \la \Gamma^\omega_{\e_2}
$$ (equal to identity on $\Gamma$),  and hence a group isomorphism 
$$\psi^{(\e_1,\e_2)}_*: \Div(\Gamma^\omega_{\e_1}) \rightarrow \Div(\Gamma^\omega_{\e_2});
\quad \quad \quad\quad
\sum_{p\in \Gamma}  n_p p \mapsto \sum_{p\in \Gamma}  n_p \psi^{(\e_1,\e_2)}(p).
$$ 
Note that $\psi^{(\e_2,\e_1)}_*$ is the inverse of $\psi^{(\e_1,\e_2)}_*$,  and that
 $\psi^{(\e_1,\e_2)}_*\circ\iota_{\e_1} = \iota_{\e_2}$;  see \eqref{iotae}.

 Note also that $\psi^{(\e_1,\e_2)}_*$ induces an isomorphism at the level of effective divisors. 
 
 We claim that  $\psi^{(\e_1,\e_2)}_*$ induces an isomorphism also at the level of principal divisors. By Remark~\ref{trivrk}, the claim implies part~\eqref{RRwc1}.
 
 To prove the claim, let $f$ be a rational  function on $\Gamma_{\e_1}^\omega$. Let $\alpha: \mathbb R \rightarrow \mathbb R$ be the homothety of ratio $\e_2/\e_1$ on $\mathbb R$, i.e., the automorphism of $\R$ given by $\alpha(x) = x\e_2/\e_1 $ for any $x\in \mathbb R$.  Define the function $\alpha\bullet  f $ on 
 $\Gamma^\omega_{\e_1}$ by   requiring that for any point of $x\in \Gamma$, $\alpha\bullet  f(x) =f(x)$, and for any point $u$ of a  virtual  loop of $\Gamma^\omega_{\e_1}$ attached at the point $v\in \Gamma$  we set
 $$\alpha\bullet  f(u) = f(v)+\alpha(f(u)-f(v)).$$
 The claim now follows by observing that $(\alpha\bullet  f)\circ \psi^{(\e_2,\e_1)}$ is a rational function on $\Gamma^\omega_{\e_2}$, and 
  $$\div((\alpha\bullet  f)\circ \psi^{(\e_2,\e_1)}) = \psi_*\, ^ {(\e_1,\e_2)}(\div(f)).$$ 

Part~\eqref{RRwc1} is proved.

To prove part \eqref{RRwc2}, recall that, as we said before, 
for the pure tropical curves ${\Gamma}_\e^{\omega} $ the   Riemann-Roch theorem holds, and
hence this part follows from the previous one.
\end{proof}
 \begin{remark}
It is clear from the proof of Theorem~\ref{RRwc} that there is no need to fix the same $\epsilon$ for all the virtual cycles. More precisely,  fix an ordering    for the virtual cycles of $G^{\omega}$  and for their edges; recall there are $\sum_{v\in V(G)} \omega(v)$ of them. Then for any $\underline{\epsilon}\in \R^{\sum \omega(v)}_{>0}$ we can define the pure tropical curve 
$\Gamma^{\omega} _{\underline{\e}}$ using ${\underline{\e}}$ to define the length on the virtual cycles in the obvious way. Then   for any $D\in \Div (\Gamma)$ the number
$r_{\Gamma_{\underline{\e}}^{\omega} }(\iota_{{\underline{\e}}}(D))$ is independent of    ${\underline{\e}}$ (where $\iota_{\underline{\e}}$ is the analog of \eqref{iotae}).
\end{remark}

\end{document}